\documentclass[review]{elsarticle}

\journal{Journal of Algebra}

\usepackage{lineno}
\modulolinenumbers[5]

\usepackage{hyperref}
\usepackage[english]{babel}
\usepackage{amsmath,amssymb,amsthm}
\usepackage{enumerate}
\usepackage{color}
\usepackage{comment}

\usepackage{pgf}
\usepackage{tikz}
\usetikzlibrary{arrows,automata}
\usepackage{verbatim}

\newtheorem{theorem}{Theorem}[section]
\newtheorem{proposition}[theorem]{Proposition}
\newtheorem{lemma}[theorem]{Lemma}
\newtheorem{corollary}[theorem]{Corollary}

\theoremstyle{definition}

\newtheorem{example}[theorem]{Example}
\newtheorem*{remark}{Remark}

\newcommand{\Z}{\mathbb{Z}}

\newcommand{\R}{\mathbb{R}}
\newcommand{\C}{\mathbb{C}}

\newcommand{\E}{\operatorname{{\bf E}}}

\renewcommand{\P}{\operatorname{\mathbb{P}}}

\begin{document}


\begin{frontmatter}

\title{Statistical Enumeration of Groups by Double Cosets \\
\vspace{3mm} \small{\emph{ In memory of Jan Saxl}}}

\author[math,stat]{Persi Diaconis}

\author[math]{Mackenzie Simper} 

\address[math]{Department of Mathematics, Stanford University}
\address[stat]{Department of Statistics, Stanford University}


\begin{abstract}
Let $H$ and $K$ be subgroups of a finite group $G$. Pick $g \in G$ uniformly at random. We study the distribution induced on double cosets. Three examples are treated in detail: 1) $H = K = $ the Borel subgroup in $GL_n(\mathbb{F}_q)$. This leads to new theorems for Mallows measure on permutations and new insights into the LU matrix factorization. 2) The double cosets of the hyperoctahedral group inside $S_{2n}$, which leads to new applications of the Ewens's sampling formula of mathematical genetics. 3) Finally, if $H$ and $K$ are parabolic subgroups of $S_n$, the double cosets are `contingency tables', studied by statisticians for the past 100 years.
\end{abstract}




\end{frontmatter}


\section{Introduction}
Let $G$ be a finite group. Pick $g \in G$ uniformly at random. What does $g$ `look like'? This ill-posed question can be sharpened in a variety of ways; this is the subject of `probabilistic group theory' initiated by Erd\H{o}s and Turan \cite{ErdosTuranI}, \cite{ErdosTuranII}, \cite{ErdosTuranIII}, \cite{ErdosTuranIV}. Specializing to the symmetric group, one can ask about features of cycles, fixed points, number of cycles, longest (or shortest) cycles, and the order of $g$ \cite{sheppLloyd}. The descent pattern has also been well-studied \cite{BDF}. Specializing to finite groups of Lie type gives `random matrix theory over finite fields' \cite{Fulman}. The enumerative theory of $p$-groups is developed in \cite{blackburn2007enumeration}. The questions also make sense for compact groups and lead to the rich world of random matrix theory  \cite{AGZmatrices}, \cite{diaconisEigenvalues}, \cite{forresterBook}. `Probabilistic group theory' is used in many ways, see \cite{dixon2002} and \cite{shalev} for alternative perspectives. 

This paper specializes in a different direction. Let $H$ and $K$ be subgroups of $G$. Then $G$ splits into double cosets and one can ask about the distribution that a uniform distribution on $G$ induces on the double cosets. Three examples are treated in detail:
\begin{itemize}
    \item If $G = GL_n(\mathbb{F}_q)$ and $H = K$ is the lower triangular matrices $B$ (a Borel subgroup), then the Bruhat decomposition
    \[
    G = \bigcup_{\omega \in S_n} B \omega B
    \]
    shows that the double cosets are indexed by permutations.  The induced measure on $S_n$ is the actively studied \emph{Mallows measure}
    \begin{equation} \label{eqn: mallows1}
    p_q(\omega) = \frac{q^{I(\omega)}}{[n]_q!},
    \end{equation}
    where $I(\omega)$ is the number of inversions in the permutation $\omega$ and $[n]_q! = (1 + q)(1 + q + q^2)\hdots (1 + q + \hdots + q^{n-1})$. The double cosets vary in size, from $1/[n]_q!$ to $q^{\binom{n}{2}}/[n]_q!$. This might lead one to think that `most $g$ lie in the big double coset'. While this is true for $q$ large, when $q$ is fixed and $n$ is large, the double coset containing a typical $g$ corresponds to an $I(\omega)$ with normal distribution centered at  $\binom{n}{2} - \frac{(n-1)}{q+1} $, with standard deviation of order $\sqrt{n}$.  See Theorem \ref{thm: inversionCLT}. The descent pattern of a typical $\omega$ is a one dependent determinantal point process with interesting properties \cite{borodinDPP}. There has been intensive work on the Mallows measure developed in the past ten years, reviewed in Section \ref{sec: mallows}. This past work focuses on $q$ as a parameter with $0 < q \le 1$. The group theory applications have $q > 1$ and call for new theory.

    \item If $G$ is the symmetric group $S_{2n}$ and $H = K$ is the hyperoctahedral group of centrally symmetric permutations (isomorphic to $C_2^n \rtimes S_n$), then the double cosets are indexed by partitions of $n$ and the induced measure is the celebrated \emph{Ewens's sampling formula}
    
    \begin{equation} \label{eqn: ewens}
    p_q(\lambda) = \frac{q^{\ell(\lambda)}}{z \cdot z_\lambda}, 
    \end{equation}
    where $\ell(\lambda)$ is the number of parts of $\lambda$, $z_\lambda = \prod_{i = 1}^n i^{a_i} a_i!$ if $\lambda$ has $a_i$ parts of size $i$, and $z = q(q + 1)\hdots (q + n - 1)$. As explained in Section 4, the usual domain for $p_q$ is in genetics. In statistical applications, $q$ is a parameter taken with $0 < q \le 1$. The group theory application calls for new representations and theorems, developed here using symmetric function theory.

    \item If $G$ is the symmetric group $S_n$ and $H = S_\lambda, K = S_\mu$ are Young subgroups corresponding to fixed partitions $\lambda$ and $\mu$ of $n$, then the double cosets are indexed by contingency tables: $I \times J$ arrays of non-negative integer entries with row sums $\lambda$ and column sums $\mu$. If $T = \{T_{ij} \}$ is such a table, the induced measure on double cosets is the \emph{Fisher-Yates distribution}
    \begin{equation} \label{eqn: fisher-yates}
    p(T) = \frac{1}{n!} \prod_{i, j} \frac{\mu_i! \lambda_j!}{T_{ij}!}.
    \end{equation}
    where $\lambda_1, \dots, \lambda_I$ are the row sums of $T$ and $\mu_1, \dots, \mu_J$ are the column sums. This measure has been well-studied in statistics because of its appearance in `Fisher's Exact Test'. This is explained in Section 5. Its appearance in group theory problems suggests new questions developed here -- what is the distribution of the number of zeros or the largest entry? Conversely, available tools of mathematical statistics (chi-squared approximation) answer natural group theory questions -- which double coset is largest, and how large is it?
\end{itemize}

The topics above have close connections to a lifetime of work by Jan Saxl. When the parabolic subgroups are $S_k \times S_{n - k}$,  the double cosets give Gelfand pairs. The same holds for $B_n \subset S_{2n}$ and, roughly, Jan prove that these are the only subgroups of $S_n$ giving Gelfand pairs for $n$ sufficiently large. He solved similar classification problems for finite groups of Lie type. These provide open research areas for the present project.

Section 2 provides background and references for double cosets, Hecke algebras, and Gelfand pairs. Section 3 treats the Bruhat decomposition $B\backslash GL_n(\mathbb{F}_q) /B$. Section 4 treats $B_n \backslash S_{2n} / B_n$ and Section 5 treats parabolic subgroups of $S_n$ and contingency tables. In each of these sections, open related problems are discussed.

\section{Background}
This section gives definitions, properties, and literature for double cosets, Hecke algebras, and Gelfand pairs. 

\subsection{Double cosets}
Let $H$ and $K$ be subgroups of the finite group $G$. Define an equivalence relation on $G$ by 
\[
s \sim t \iff s = h^{-1}t k \,\,\,\,\,\,\,\, \text{for} \,\,\,\,\,\,\,\, s, t \in G, \,\,\, h \in H, \,\,\, k \in K.
\]
The equivalence classes are called \emph{double cosets}, written $HsK$ for the double coset containing $s$ and $H\backslash G/K$ for the set of double cosets. This is a standard topic in undergraduate group theory \cite{suzuki}, \cite{dummit2004abstract}. A useful development is in \cite{carter}, Section 2.7. Simple properties are:
\begin{align}
    & |HsK| = \frac{|H||K|}{|H \cap sKs^{-1}|} = \frac{|H||K|}{|K \cap s^{-1}Hs|} \label{eqn: prop1} \\
    & |G : H| = \sum_{HsK \in H \backslash G /K} \frac{|HsK|}{|H|} \label{eqn: prop2}  \\
    & |H \backslash G / K| = \frac{1}{|H||K|} \sum_{h \in H, k \in K} |G_{hk}|, \,\,\,\,\,\, \text{where} \,\,\,\,\, G_{hk} = \{g: h^{-1} g k = g \}. \label{eqn: prop3} 
\end{align}

Despite these nice formulas, enumerating the number of double cosets can be an intractable problem. For example, when $H$ and $K$ are Young subgroups, double cosets are contingency tables with fixed row and column sums. Enumerating these is a \#-P complete problem \cite{DG}.

Consider the problem of determining the smallest (and largest) double coset. If $H = K$, the smallest is the double coset containing $id$ (with size $|H \cdot H| = |H|$). When $K \neq H$, it is not clear. Is it the double coset containing $id$? Not always. Indeed, for $H$ a proper subgroup of $G$, let $g$ be any element of $G$ not in $H$. Let $K = H^g = g^{-1} H g$. Then $HgK = Hgg^{-1}Hg = HHg = Hg$, so the double coset $HgK$ is just a single right coset of $H$. This has minimal possible size among double cosets $H \times K$. Since $g$ is not in $H$, $HgK = Hg$ does not contain the identity. It can even happen that the double coset containing the identity has maximal size. This occurs, from \eqref{eqn: prop1} above, whenever $H \cap K = id$. 

For the largest size of a double coset, from \ref{eqn: prop1} note that this can be at most $|H| \cdot |K|$. If $H = K$ this occurs if $H$ is not normal. For $H$ and $K$ different, it is not clear. For parabolic subgroups of $S_n$ a necessary and sufficient condition for the maximum size to be achieved is that $\lambda$ majorizes $\mu^t$. See \cite{jamesKerber} Section 1.3, which shows that the number of double cosets achieving the maximum is the number of $0/1$ matrices with row sums $\lambda$ and column sums $\mu^t$.

The seemingly simple problem of deciding when there is only one double coset becomes the question of factoring $G = HK$. This has a literature surveyed in \cite{Baumeister}.

All professional group theorists use double cosets -- one of the standard proofs of the Sylow theorems is based on \eqref{eqn: prop2}, and Mackey's theorems about induction and restrictions are in this language. In addition, double cosets have recently been effective in computational group theory. Laue \cite{laue} uses them to enumerate all isomorphism classes of semi-direct products. Slattery \cite{slattery} uses them in developing a version of coset enumeration.

\subsection{Hecke algebras}

With $H$ and $K$ subgroups of a finite group $G$, consider the group algebra over a field $\mathbb{F}$:
\[
L_\mathbb{F}(G) = \{f: G \to \mathbb{F} \}.
\]
This is an algebra with $(f_1 + f_2)(s) = f_1(s) + f_2(s)$ and $f_1 \ast f_2(s) = \sum_t f_1(t) f_2(st^{-1})$. The group $H \times K$ acts on $L_\mathbb{F}(G)$ by
\[
f^{h, k}(s) = f(h^{-1} s k).
\]
The bi-invariant functions (satisfying $f(h^{-1} s k) = f(s)$ for all $h \in H, k \in K, s \in G$) form a sub-algebra of $L_\mathbb{F}(G)$ which is here called the \emph{Hecke algebra}. Many other names are used, see \cite{solomon} for history.

Hecke algebras are a mainstay of modern number theory (usually with infinite groups). They are also used by probabilists (e.g.\ \cite{DR}) and many stripes of algebraists. Curtis and Reiner \cite{curtisReiner} is a standard reference for the finite theory. We denote them by $L_\mathbb{F}(H \backslash G/ K)$. Clearly the indicator functions of the double cosets form a basis for $L_\mathbb{F}(H \backslash G/K)$.

\subsection{Gelfand pairs}

For some choices of $G$ and $H$, with $K = H$, the space $L_\mathbb{F}(H \backslash G/K)$ forms a \emph{commutative} algebra (even though $G$ and $H$ are non-commutative). Examples with $G = S_n$ are $H_n = S_k \times S_{n - k}$ or $H_n = B_n$ in $S_{2n}$. Of course, $G$ acts on $L_\mathbb{F}(G/H)$ (say with $\mathbb{F} = \C$) and commutativity of $L_\C(G/H)$ is equivalent to the representation of $G$ on $L_\C(G/H)$ being multiplicity free: Since $L_\C(G/H) = \text{Ind}_H^G(1)$ (the trivial representation of $H$ induced up to $G$), Frobenius reciprocity implies that each irreducible $\rho_\lambda$ occurring in $L_\C(G/H)$ has a 1-dimensional subspace of left $H$-invariant functions. Let $s_\lambda$ be such a function, normalized by $s_\lambda(1) = 1$. These are the \emph{spherical functions} of the Gelfand pair $(G, H)$. Standard theory shows that the spherical functions $\{s_\lambda \}$ form a second basis for $L_\C(H\backslash G/H)$.

We will not develop this further and refer to \cite{diaconisBook} (Chapter 3F), \cite{CSTTbook}, \cite{Letac} for applications of Gelfand pairs in probability. 

We also note that Gelfand pairs occur more generally for compact and non-compact groups. For example, $\mathcal{O}_n/\mathcal{O}_{n-1}$ is Gelfand and the spherical functions become the spherical harmonics of classical physics. For $\mathcal{O}_n \subset GL_n(\R)$, the spherical functions are the zonal polynomials beloved of older mathematical statisticians. Gelfand pairs are even useful for large groups such as $S_\infty$ and $U_\infty$, which are not locally compact. See \cite{borodinOlshanki}.

Clearly, finding subgroups $H$ giving Gelfand pairs is a worthwhile project. Jan Saxl worked on classifying subgroups giving Gelfand pairs over much of his career \cite{saxl}, \cite{saxlInglisLiebeck}. He gave definitive results for the symmetric and alternating groups and for most all the almost simple subgroups of Lie type. Alas, it turns out that if $n$ is sufficiently large, then $S_k \times S_{n - k}$ and $B_n$ give the only Gelfand pairs in $S_n$ (at least up to subgroups of index $2$; e.g.\ $A_k \times S_{n-k}$ is Gelfand in $S_n$).

\section{Bruhat decomposition and Mallows measure}

\subsection{Introduction} \label{sec: bruhatIntro}
Let $G = GL_n(\mathbb{F}_q)$, the general linear group over a field with $q$ elements. Let $B$ be the lower triangular matrices in $G$. Let $W$ denote the permutation group embedded in $G$ as permutation matrices. The decomposition of $G$ into $B-B$ double cosets is called the \emph{Bruhat decomposition} \cite{solomon} and has the following properties:
\begin{equation} \label{eqn: Bruhat}
    G = \bigcup_{\omega \in W} B\omega B, \,\,\,\,\,\,\, |B| = (q - 1)^n q^{\binom{n}{2}}, \,\,\,\,\,\,\, |G| = |B| \prod_{i = 1}^{n-1}(1 + q + \hdots + q^i).
\end{equation}
Thus, permutations index the double cosets. The size of $BwB$ is
\[
|B\omega B| = |B| q^{I(\omega)},
\]
where $I(\omega)$ is the number of inversions of $\omega$ (that is, $I(\omega) =| \{i < j: \omega_i > \omega_j \} |$). Dividing by $|G|$, we get the induced measure
\begin{equation} \label{eqn: mallows}
    p_q(\omega) = \frac{q^{I(\omega)}}{[n]_q!}, \,\,\,\,\,\,\, [n]_q! = \prod_{i = 1}^{n-1}(1 + q + \hdots + q^i)
\end{equation}

\begin{example}
In $S_3$, the inversions are

\begin{center}
\begin{tabular}{c|c c c c c c}
    $\omega$ & 123 & 132 & 213 & 231 & 312 & 321  \\  \hline
     $I(\omega)$ & 0 & 1 & 1 & 2 & 2 & 3
\end{tabular}
\end{center}
and $(1 + q)(1 + q + q^2) = 1 + 2 q + 2q^2 + q^3$.
\end{example}

The measure $p_q(\omega)$, $\omega \in W$, is studied as the \emph{Mallows measure} on $W = S_n$ in the statistical and combinatorial probability literature. A review is in Section \ref{sec: mallows}. Much of this development is for the statistically natural case of $0 < q < 1$ with $q$ close to $1$. The group theory application has $q = p^a$ for a prime $p$ and $a \in \{1, 2, 3, \dots \}$. This calls for new theorems and insights. The question of interest is
\begin{equation}
    \text{Pick } \, g \in G \,\, \text{from the uniform distribution. What double coset is } \, g \, \text{ likely to be in?}
\end{equation}

An initial inspection of \eqref{eqn: mallows} reveals the minimum and maximum values: $p_q(id) = 1/[n]_q!$ and $p_q(\omega_0) = q^{\binom{n}{2}}/[n]_q!$ for $\omega_0 = n(n-1)\hdots 21$, the reversal permutation. Thus, $B \omega_0 B$ is the largest double coset. It is natural to guess that `maybe most elements are in $B \omega_0 B$'. This turns out to not be the case.

\begin{lemma}
\begin{equation}
    \frac{q^{\binom{n}{2}}}{[n]_q!} = c(q) \left( 1 - \frac{1}{q} \right)^{n-1}, \,\,\,\,\,\,\, c(q)= \prod_{i = 2}^{n - 2} \left(1 - \frac{1}{q^i} \right)^{-1}
\end{equation}
\end{lemma}

\begin{proof}
Using that $\binom{n}{2} = \sum_{i = 1}^{n-1} i$, simple algebra gives
\begin{align*}
    \frac{q^{\binom{n}{2}}}{[n]_q!} &= \frac{q^{\binom{n}{2}}}{(1 + q)(1 + q + q^2) \hdots (1 + q + \hdots + q^{n-1})} \\
    &= \frac{1}{\left(1 + \frac{1}{q} \right) \left(1 + \frac{1}{q} + \frac{1}{q^2} \right) \hdots \left(1 + \frac{1}{q} + \hdots + \frac{1}{q^{n-1}} \right)} \\
    &= \frac{\left( 1 - \frac{1}{q} \right)^{n-1}}{\prod_{i = 2}^{n-2} \left(1 - \frac{1}{q^i} \right)}
\end{align*}

\end{proof}

The infinite product $\prod_{i = 1}^\infty (1 - 1/q^i)$ converges. This shows that for fixed $q$, when $n$ is large $p_q(\omega_0)$ is exponentially small. Of course, for $n$ fixed and $q$ large, $p_q(\omega_0)$ tends to $1$ (only $q \gg n$ is needed). 

In Section \ref{sec: cycles}, it is shown that a uniform $g$ is contained in $B \omega B$ for $I(\omega) = \binom{n}{2} - \frac{(n-1)}{q - 1}  +Z \frac{\sqrt{(n-1)q}}{q - 1}$ with $Z$ a standard normal random variable. 

Let us conclude this introductory section with two applied motivations for studying this double coset decomposition.

\begin{example}[LU decomposition of a matrix]
Consider solving $Ax = b$ with $A$ fixed in $GL_n(\mathbb{F}_q)$ and $b$ fixed in $\mathbb{F}_q^n$. The standard `Gaussian elimination' solution subtracts an appropriate multiple of the first row from lower rows to  make the first column $(1, 0, \dots, 0)^T$, then subtracts multiples of the second row to make the second column $(*, 1, 0, \dots, 0)^T$, and so on, resulting in the system
\[
Ux^* = b^*
\]
with $U$ upper triangular. This can be solved inductively for $x^*$ and then $x$. This description assumes that at stage $j$, the $(c, j)$ entry of the current triangularization is non-zero. If it \emph{is} zero, a permutation (pivoting step) is made to work with the first non-zero element in column $j$. A marvelous article by Roger Howe \cite{howe} shows in detail how this is equivalent to expressing $A = B \omega B$ with the number of pivoting steps being $q^{n - I(\omega)}$. Thus, matrices in the largest Bruhat cell require no pivots and $p_q(\omega)$ gives the chance of various pivoting permutations.
\end{example}

\begin{example}[Random generation for $GL_n(\mathbb{F}_q)$]
Suppose one wants to generate $N$ independent picks from the uniform distribution on $GL_n(\mathbb{F}_q)$. We have had to do this in cryptography applications when $q = 2, n = 256, N = 10^6$. Testing conjectures for $G$ also uses random samples. One easy method is to fill in an $n \times n$ array with independent picks from the uniform distribution on $\mathbb{F}_q$ and then check if the resulting matrix $A$ is invertible (using Gaussian elimination). If $A$ is not invertible, this is simply repeated. The chance of success is approximately $\prod_{i = 1}^\infty \left(1 - \frac{1}{q^i} \right)$ ($\approx 0.29$ when $q = 2$). Alas, this calls for a variable number of steps and made a mess in programming our crypto chip. 

Igor Pak suggested a simple algorithm that works in one pass:
\begin{enumerate}
    \item Pick $\omega \in W$ from $p_q(\omega)$.
    
    \item Pick $B_1, B_2 \in B$ uniformly.
    
    \item Form $B_1 \omega B_2$.
\end{enumerate}

Since picking $B_i$ uniformly is simple, this is a fast algorithm. But how to pick $\omega$ from $p_q$? The following algorithm is standard:

\begin{enumerate}
    \item Place symbols $1, 2, \dots, n$ down in a row sequentially, beginning with $1$.
    
    \item If symbols $1,2, \dots, i - 1$ have been placed, then place symbol $i$ leftmost with probability $q^{i-1}(q - 1)/(q^i - 1)$, secondmost with probability $q^{i-2}(q - 1)/(q^i - 1)$, \dots and $i$th with probability $(q - 1)/(q^i - 1)$.
    
    \item Continue until all $n$ symbols are placed.
\end{enumerate}

\end{example}

The following sections develop some theorems for the Mallows distribution \eqref{eqn: mallows} for $q > 1$ fixed and $n$ large. In Section \ref{sec: cycles}, the normality of $I(\omega)$ is established. Section \ref{sec: mallows} develops other properties along with a literature review of what is known for $q < 1$. The descent pattern is developed in \ref{sec: descents}, generalizations to other finite groups and parallel orbit decompositions (e.g.\ $G \times G$ acting on $Mat(n, q)$) are in Section \ref{sec: 5conclusion}. These sections are also filled with open research problems.

\subsection{Distribution of $I(\omega)$} \label{sec: cycles}

This section proves the limiting normality of the number of inversions $I(\omega)$ under the Mallows measure $p_q(\omega)$ defined in \eqref{eqn: mallows}, when $q > 1$ is fixed and $n$ is large. Thus, most $g \in GL_n(\mathbb{F}_q)$ are \emph{not} in the largest double coset. 

\begin{theorem} \label{thm: inversionCLT}
With notation as above, for any $x \in \R$,
\[
p_q \left\lbrace \frac{I(\omega) - \binom{n}{2} + \frac{(n-1)}{q - 1}}{\sqrt{(n-1)q}/(q-1)} \le x \right\rbrace = \frac{1}{\sqrt{2 \pi}} \int_{- \infty}^x e^{-t^2/2} \, dt + o(1).
\]
The error term is uniform in $x$.
\end{theorem}

\begin{proof}
The argument uses the classical fact that under $p_q$ on $S_n$, $I(\omega)$ is exactly distributed as a sum of independent random variables. Let $P_j(i) = q^i(q - 1)/(q^{j+1} - 1)$ for $0 \le i \le j < \infty$. Write $X_j$ for a random variable with distribution $P_j, 1 \le j \le n - 1$, taking $X_j$ independent. Then
\begin{equation} \label{eqn: Iomega}
    p_q(I(\omega) = a) = \P \left\lbrace X_1 + X_2 + \hdots + X_{n-1} = a \right\rbrace \,\,\,\,\,\,\, \text{for all } \,\, n \,\,\, \text{and} \,\,\, 0 \le a \le n - 1.
\end{equation}
To see \eqref{eqn: Iomega}, use generating functions. Rodrigues \cite{rodrigues1839note} proved for any $\theta$ that
\[
\sum_{\omega \in S_n} \theta^{I(\omega)} = (1 + \theta)(1 + \theta + \theta^2) \hdots (1 + \theta + \hdots + \theta^{n-1}).
\]
Take $\theta = xq$ and divide both sides by $[n]_q!$ to see
\[
\E_q[x^{I(\omega)}] = \E[x^{X_1}] \hdots \E[x^{X_{n-1}}].
\]
Under $P_j$
\[
P_j(X_{j} = j - a) = \frac{q^{j  - a}(q - 1)}{q^{j+1} - 1} = \left( 1 - \frac{1}{q} \right) \frac{1}{q^a} \left(1 + \frac{1}{q^{j+1} - 1} \right).
\]
Thus, when $j$ is large (and using that $q > 1$), the law of $j - X_j$ is exponentially close to a geometric random variable $X$ with $P(X = a) = \left(1 - \frac{1}{q} \right) \left( \frac{1}{q} \right)^a, 0 \le a < \infty$. This $X$ has $\E[X] = 1/(q - 1)$ and $\text{Var}(X) = q/(q - 1)^2$. Now, the classical central limit theorem implies the result.
\end{proof}

\subsection{Simple properties of $p_q(\omega)$} \label{sec: mallows}

The discussion above points to the question of: What properties of $\omega$ are `typical' under $p_q(\omega)$? We now see that $\omega$ with $I(\omega) = \binom{n}{2} - (n-1)/(q - 1) \pm \frac{\sqrt{(n-1)q}}{q - 1}$ are typical, but are all such $\omega$ equally likely?

The distribution $p_q(\omega)$ is studied (for general Coxeter groups) in \cite{DR}. They show (for all $q, n$)
\begin{align}
    & p_q(\omega) = p_q(\omega^{-1}) \label{eqn: pq1} \\
    & p_q(\omega_1 = j) = q^{j-1}(q - 1)/(q^2 - 1) \label{eqn: pq2} \\
    & p_q(\omega_n = j) = q^{n - j}(q - 1)/(q^n - 1) \label{eqn: pq3} \\
    &p_q(\omega) = p_{q^{-1}}(R(\omega)), \label{eqn: pq4}
\end{align}
where in the last expression $R(\omega)$ is the reversal of $\omega$ (e.g.\ $R(31542) = 24513$). However, there do not appear to be simple expressions for $p_q(\omega_i = j)$, $1 < i < n$, nor for the distribution of the number of fixed points, cycles, or other features standard in enumerative combinatorics.

There has been remarkable study of features when $q$ is close to $1$ (often $q = 1 - \beta/n$). These include
\begin{itemize}
    \item The limiting distribution of the empirical measure $\frac{1}{n} \sum \delta_{i, \omega_i}$ was studied by Shannon Starr \cite{starr2009}. He shows, for $q = 1 - \beta/n$,
    \[
    \lim_{\epsilon \downarrow 0, n \to \infty} p_{q} \left\lbrace \left| \frac{1}{n} \sum f \left( \frac{i}{n}, \frac{\omega_i}{n} \right) - \int_{[0, 1] \times [0, 1]} f(x, y) u(x, y) \, dx dy \right| > \epsilon \right\rbrace = 0
    \]
    for any continuous function $f: [0, 1] \times [0, 1] \to \R$, where
    \[
    u(x, y) = \frac{(\beta/2) \sinh(\beta/2)}{\left(e^{\beta/4} \cosh(\beta(x - y)/2) - e^{-\beta/4} \cosh(\beta(x + y - 1)/2) \right)}.
    \]
    
    Starr derives these results rigorously by considering a Gibbs measure on permutations $Z^{-1}(\beta) e^{-\beta H(\omega)}$ with $H(\omega) = \frac{1}{n-1} \sum_{1 \le i < j \le n} \delta_{(0, \infty)}(\omega_j - \omega_i)$. This work is continued in \cite{starrWalters2018}, bringing in fascinating connections with statistical mechanics. 
    
    The function $u(x, y)$ above is an example of a permutation limit or `permuton'. It figures in many of the developments below.
    
    \item Starr's work, along with the emerging world of permutation limit theory \cite{hoppenKohayakawaMoreira}, \cite{borga}, is central to the work of Bhattacharya and Mukherjee \cite{bhattMukerjee}, who study the degree distribution of the permutation graphs (an edge from $i$ to $j$ if and only if $i < j$ and $\omega(i) > \omega(j)$) under $p_{1 - \beta/n}(\omega)$, \cite{bhattMukerjee}.
    
    \item Remarkable work on statistical aspects of the Mallows model (given an observed $\omega$, how do you estimate $\beta$ and how do such estimates behave?) is in \cite{mukherjeeI}. \nocite{mukherjeeII} This work treats other Mallows models of the form $p_\beta(\omega) = e^{(-\beta/n) d(\omega, id)}$, for $d$ a metric on $S_n$.
    
    \item In \cite{mukherjeeII}, for $q = 1 - \beta/n$ with $\beta$ positive or negative (but fixed), Mukherjee extends \eqref{eqn: pq2}, \eqref{eqn: pq3}, \eqref{eqn: pq4} above to
    \[
    \lim_{n \to \infty} \left| \frac{p_q(\sigma(i) = j)}{u(i/n, j/n)/n} - 1 \right| = 0
    \]
    with $u$ the permutation limit above. He gives similar results for several coordinates and uses these to prove limit theorems for the number of fixed points and cycles. 
    
    \item The number $d(\omega)$ of descents in $\omega$ (and indeed the distribution of $d(\omega) + d(\omega^{-1})$) is shown to be approximately normally distributed in \cite{he2020central} for $q = 1 - \beta/n$. A description of $d(\omega)$ for fixed $q > 1$ is in Section \ref{sec: descents} below. This also follows from \cite{he2020central}, with a Barry-Esseen quality error. 
    
    \item The cycles of $\omega$ have limit distributions determined by \cite{gladkichPeledCycle} for $q = 1 - \beta/n$.
    
    \item The length of the longest increasing subsequence under $p_{1 - \beta/n}$ has a fascinating limit theory, see \cite{muellerStarr}, \cite{bhatnagarPeled}. See \cite{basuMonotoneSubsequences} for work on the longest monotone subsequence, which works for fixed $0 < q < 1$ and thus is relevant to the present group theory applications. This paper develops a probabilistic regenerative process for building permutations from the Mallows model that should be broadly useful for further distribution questions.

    \item  Three further papers develop properties of the Mallows model for fixed $q$ less than $1$ (and are hence applicable to fixed $q$ greater than $1$ because of \eqref{eqn: pq4}). In \cite{craneDesalvoSergi}, the authors study the distribution of pattern avoiding permutations for local patterns (e.g.\ $321$ avoiding) and related topics. A variety of techniques are used. These may well transfer to other statistics. In \cite{gnedinOshanki}, \cite{gnedinOlshanki2} a limiting measure on permutations of the natural numbers (and $\Z$) is introduced as a limit of finite Mallows measures. As one offshoot, the limiting distribution of $\sigma_i - i$ (and multivariate extensions) is determined. The main interest of these papers is a $q$-analog of de Finetti's Theorem.
\end{itemize}

\subsection{Descents} \label{sec: descents}

In this section, $q > 1$ is fixed. A permutation $\omega \in S_n$ has a \emph{descent} at position $i$ if $\omega_{i+1} < \omega_i$. The total number of descents is $d(\omega)$. The \emph{descent set} is $S(\omega) = \{i \le n - 1 : \omega_{i + 1} < \omega_i \}$. For example, $\omega = 561432$ has $d(\omega) = 3$ and descent set $S(\omega) = \{2, 4, 5 \}$. Descents are a basic descriptive statistic capturing the `up/down' pattern in a permutations. The distribtuion theory of $d(\omega)$ is classical combinatorics going back to Euler. Descent sets also have a remarkable combinatorial structure, see \cite{BDF} for an overview.

Recent work allows detailed distribution theory for $d(\omega)$ and $S(\omega)$ under the Mallows distribution $p_q(\omega)$ for fixed $q > 1$. To describe this, let 
\[
X_i(\omega) = \begin{cases} 1 \,\,\,\,\,\, \text{if} \,\,\, \omega_{i + 1} < \omega_i \\
0 \,\,\,\,\,\, \text{if} \,\,\, \omega_{i + 1} > \omega_i
\end{cases}, \,\,\,\,\,\,\, 1 \le i \le n -1.
\]
If $\omega$ is random, then $X_1, X_2, \dots, X_{n-1}$ is a point process. The following result of Borodin, Diaconis, Fulman \cite{BDF} describes many properties of this process. For further definitions and background, see \cite{borodinDPP}.

\begin{theorem} \label{thm: descentThm}
Let $p_q$, $q > 1$, be the Mallows measure \eqref{eqn: mallows}.

\begin{enumerate}[(a)]
    \item The chance that a random permutation chosen from $p_q$ has descent set containing $s_1 < s_2 < \hdots < s_k$ is
    \[
    \text{det} \left[ \frac{1}{[s_{j+1} - s_j]_q !} \right]_{i, j = 0}^k,
    \]
    with $s_0 = 0, s_{k + 1} = n$.
    
    \item The point process $X_i(\omega)$ is stationary, one dependent, and determinantal with kernel $K(x, y) = k(x - y)$ where
    \[
    \sum_{m \in \Z} k(m) z^m = \frac{1}{1 - \left(\sum_{m = 0}^\infty z^n/ [m]_q! \right)^{-1}}.
    \]
    
    \item The chance of finding $k$ descents in a row is $q^{\binom{k + 1}{2}}/[k + 1]_q!$. In particular, the number $d(\omega)$ of descents has mean $\mu(n, q) = \frac{q}{q+1}(n - 1)$ and variance
    \[
    \sigma^2  = q \left( \frac{(q^2 - q + 1)(n - q^2 + 3q - 1)}{(q + 1)^2(1 + q + q^2)} \right).
    \]
    Normalized by its mean and variance, the number of descents has a limiting standard normal distribution.
\end{enumerate}
\end{theorem}

\textbf{Remarks}

\begin{enumerate}
    \item Consider the distribution of $d(\omega)$ in part (c) of Theorem \ref{thm: descentThm}. Under the uniform distribution on $S_n$, $d(\omega)$ has mean $(n - 1)/2$ and variance $(n + 1)/12$ (obtained by setting $q = 1$ in the formula in part (c)). The distribution $p_q$ pushes $\omega$ toward $\omega_0$. How much? The mean increase to $\frac{q}{q+1}(n - 1)$ and, as makes sense, the variance \emph{decreases}. For large $q$, the mean goes to the maximum value $(n - 1)$ and the variance goes to zero.
    
    \item The paper \cite{BDF} gives simple formulas for the $k$-point correction function $p_q(S(\omega) \subset S)$ for general sets $S$.
    
    \item There is an interesting alternative way to compute various moments for $d(\sigma)$ under the measure $p_q$. Let
    \[
    A_n(y) = \sum_{\sigma \in S_n} p_q(\sigma) y^{d(\sigma)}.
    \]
    \cite{stanley} gives
    \[
    \sum_{n = 0}^\infty A_n(y) z^n = \frac{1 - y}{E(z(y - 1)) - y}, \,\,\,\,\,\,\,\,\,\, E(w) = \sum_{n = 0}^\infty \frac{q^{\binom{n}{2}}}{[n]_q!} w^n.
    \]
    Differentiating in $y$ and setting $y = 1$ gives the generating function of the falling factorial moments for $d$ under $p_q$. Using Maple, Stanley (personal communication) computes
    \begin{align*}
        &\sum_n A_n'(1) z^n = \frac{q z^2}{(z - 1)^2(q + 1)} \\
        &\sum_n A_n''(1) z^n = \frac{q^2 z^3(q^2 + q + z)}{(z - 1)^3(q + 1)^2(q^2 + q + 1)} \\
        &\sum_n A_n'''(1) z^n = \frac{q^3 z^4 (q^4 + q^3 + 2q^2z - qz^2 + 2qz + z^2)}{(z - 1)^4(q + 1)^3(q^2 + q + 1)(q^2 + 1)}
    \end{align*}
    These give independent checks on the mean and variance reported before and an expression for the third moment. It would be a challenge to prove the central limit theorem by this route.
\end{enumerate}

Sections \ref{sec: bruhatIntro}-\ref{sec: descents} underscore our main point: Enumeration by double cosets can lead to interesting mathematics.

\subsection{Other groups and actions} \label{sec: 3conclusion}

The Bruhat decomposition \eqref{eqn: Bruhat} is a special case of more general results. Bruhat showed that a classical semi-simple Lie group has a double coset of this form where $B$ is a maximal solvable subgroup of $G$ and $W$ is the Weyl group. Then Chevalley showed the construction makes sense for any field, particularly finite fields. This gives
\[
|G| = |B| \sum_{\omega \in W} q^{\ell(\omega)},
\]
with $\ell(\omega)$ the length of the word $\omega$ in the Coxeter generators. The length generation function factors
\[
\sum_{\omega \in W} q^{\ell(\omega)}  = \prod_{i = 1}^n ( 1 + q + \hdots + q^{e_i} ),
\]
where $e_i$, the exponents of $W$, are known. From here, one can prove the analog of Theorem \ref{thm: inversionCLT}. The Weyl groups have a well developed descent theory (number of positive roots sent to negative roots) and one may ask about the analog of Theorem \ref{thm: descentThm}, along with the other distribution questions above.

We want to mention two parallel developments. Louis Solomon \cite{solomon} has built a beautiful parallel theory for describing the orbits of $GL_n(\mathbb{F}_q) \times GL_n(\mathbb{F}_q)$ on $Mat(n, q)$, the set of $n \times n$ matrices. This has been wonderfully developed by Tom Halverson and Arun Ram \cite{halvseronRam}. None of the probabilistic consequences have been worked out. There is clearly something worthwhile to do.

Second, Bob Gualnick \cite{guralnick} has classified the orbits of $GL_n(\mathbb{F}_q) \times GL_m(\mathbb{F}_q)$ acting on the set $Mat(n, m; q^2)$ of $n \times m$ matrices over $\mathbb{F}_q$. Estimating the sizes and other natural questions about the orbit in the spirit of this section seems like an interesting project.

Finally, it is worth pointing out that finding `nice descriptions' of double cosets is usually \emph{not} possible. For example, let $U_n(q)$ be the group of $n \times n$ uni-upper triangular matrices with entries in $\mathbb{F}_q$. Let $G = U_n(q) \times U_n(q)$, with $H = K = U_n(q)$ embedded diagonally. Describing $H-H$ double cosets is a well-studied wild problem in the language of quivers \cite{gabrielRoiter}. In \cite{20author}, this was replaced by the easier problem of studying the `super characters' of $U_n$. This leads to nice probabilistic limit theorems. See \cite{CDKRI}, \cite{CDKRII}.

\section{Hyperoctahedral double cosets and the Ewens sampling formula}

\subsection{Introduction}
Let $B_n$ be the group of symmetries of an $n$-dimensional hypercube. This is one of the classical groups generated by reflections. It can be represented as
\begin{equation}
    B_n \cong C_2^n \rtimes S_n
\end{equation}
with $S_n$ acting on the binary $n$-tuples by permuting coordinates. Thus, $|B_n| = n! 2^n$. For present purposes it is useful to see $B_n \subset S_{2n}$ as the subgroup of centrally symmetric permutations. That is, permutations $\sigma \in S_{2n}$ with $\sigma(i) + \sigma(2n + 1 - i) = 2n + 1$ for all $1 \le i \le n$. For example, when $n = 2$ we have $|B_2| = 8$ and, as elements of $S_4$, can write
\[
B_{n} = \{1234, 4231, 1324, 4321, 3142, 2143, 3412, 2413 \}.
\]
The first and last values in each permutation sum to five, as do the middle two values. This representation is useful in studying perfect shuffles of a deck of $2n$ cards \cite{DGK}.

The double coset space $B_n \backslash S_{2n} / B_n$ is a basic object of study in the statisticians world of zonal polynomials. Macdonald (\cite{macdonald}, Section 7.1) develops this clearly, along with citations to the statistical literature, and this section follows his notation.

We begin by noting two basic facts: 1) The double cosets $B_n \backslash S_{2n} / B_n$ form a Gelfand pair. 2) The double cosets are indexed by partitions of $n$. To see how this goes, to each permutation $\sigma \in S_{2n}$ associate a graph $T(\sigma)$ with vertices $1, 2, \dots, 2n$ and edges $\{\epsilon_i, \epsilon_i^\sigma \}_{i = 1}^n$ where $\epsilon_i$ joins vertices $2i - 1, 2i$ and $\epsilon_i^\sigma$ joins vertices $\sigma(2i - 1), \sigma(2i)$. Color the $\epsilon_i$ edges red and the $\epsilon_i^\sigma$ edges as blue. Then, each vertex lies on exactly one red and one blue edge. This implies the components of $T(\sigma)$ are cycles with alternating red and blue edges, so each cycle has an even length. Dividing these cycle lengths by $2$ gives a partition of $n$, call it $\lambda_\sigma$. 

\begin{example}
Take $n = 3$ and $\sigma = 612543$. The graph $T(\sigma)$ is 

\begin{center}
    
\begin{tikzpicture}[->,>=stealth',shorten >=1pt,auto,node distance=2.8cm,
                    semithick]
  \tikzstyle{every state}=[fill=none,draw=black,text=black]

  \node[state] (A)                    {$1$};
  \node[state]         (B) [right of=A] {$2$};
  \node[state]         (C) [right of=B] {$3$};
  \node[state]         (D) [below of=C] {$4$};
  \node[state]         (E) [left of=D]       {$5$};
    \node[state]         (F) [left of=E]       {$6$};

        
\draw[red] (A) edge [-] node [above, text = black] {$\epsilon_1$} (B);
\draw[red] (C) edge [bend right, -] node [left, text = black] {$\epsilon_2$} (D); 
\draw[red] (E) edge [-] node [above, text = black] {$\epsilon_3$} (F); 

\draw[blue] (A) edge [-] node [right, text = black] {$\epsilon_1^\sigma$} (F); 
\draw[blue] (B) edge [-] node [right, text = black] {$\epsilon_2^\sigma$} (E); 
\draw[blue] (C) edge [bend left, -] node [right, text = black] {$\epsilon_3^\sigma$} (D); 

\end{tikzpicture}

\end{center}

Here there is a cycle of length $4$ and a cycle of length $2$, thus this corresponds to the partition $\lambda_\sigma = (2, 1)$.
\end{example}

Macdonald proves (2.1 in Section 7.2, \cite{macdonald}) that $\lambda_\sigma = \lambda_{\sigma'}$ if and only if $\sigma \in B_n \sigma' B_n$. Thus, the partitions of $n$ serve as double coset representatives for $B_n \backslash S_{2n} / B_n$.

If we denote $B_\lambda = B_n \sigma B_n$, $\lambda = \lambda_\sigma$, then
\begin{equation} \label{eqn: Blam}
    |B_\lambda| = \frac{|B_n|^2}{2^{\ell(\lambda)} z_\lambda},
\end{equation}
with $z_\lambda = \prod_{i = 1}^n i^{a_i} a_i!$ and $\ell(\lambda) = \sum_{i = 1}^n a_i(\lambda)$ is the number of parts in $\lambda$, where $\lambda$ has $a_i$ parts of size $i$. For example, for $\sigma = id$, we see $\lambda_\sigma = 1^n, z_\lambda = n!$ and $|B_{1^n}| = (2^n n!)^2/(2^n n!) = |B_n|$. The largest double coset corresponds to the $2n$ cycles $(12 \hdots 2n)$ in $S_{2n}$ (not all $2n$ cycles are in the same double coset).

To see this, let $f(\lambda) = z_\lambda 2^{\ell(\lambda)}$ for a partition $\lambda$ of $n$. Note that for \emph{any} $\lambda$, if a box from the lower right corner is moved to the right end of the top row, the result $\lambda'$ is still a partition. For example,

\begin{center}
\begin{tikzpicture}
\draw (-.5, -.5) node {$\lambda = $ };
\foreach \x in {0,...,3}
	\filldraw (\x*.25, 0) circle (.5mm);
\foreach \x in {0,...,3}
	\filldraw (\x*.25, -.25) circle (.5mm);
\foreach \x in {0,...,2}
	\filldraw (\x*.25, -.5) circle (.5mm);
\foreach \x in {0,...,0}
	\filldraw (\x*.25, -.75) circle (.5mm);
	\filldraw [red] (1*.25, -.75) circle (.5mm);

\draw (2, -.5) node {$\lambda^\prime = $ };
\foreach \x in {0,...,3}
	\filldraw (\x*.25 + 2.75, 0) circle (.5mm);
	\filldraw [red] (4*.25 + 2.75, 0) circle (.5mm);
\foreach \x in {0,...,3}
	\filldraw (\x*.25 + 2.75, -.25) circle (.5mm);
\foreach \x in {0,...,2}
	\filldraw (\x*.25 + 2.75, -.5) circle (.5mm);
\foreach \x in {0,...,0}
	\filldraw (\x*.25 + 2.75, -.75) circle (.5mm);

\end{tikzpicture}
\end{center}

\begin{lemma}
With notation above, $f(\lambda) > f(\lambda')$.
\end{lemma}

\begin{proof}
Assume the first row of $\lambda$ has $a$ boxes and the last row has $b$ boxes. Consider $f(\lambda')/f(\lambda)$. If $b > 1$, then $\ell(\lambda') = \ell(\lambda)$. With $z_\lambda = \prod_{i = 1}^n i^{m_i} m_i!$, where $m_i$ is the number of parts of $\lambda$ of length $i$, then from $\lambda$ to $\lambda'$ only $i = a, b, a + 1, b - 1$ will change. Thus,
\begin{align*}
\frac{f(\lambda')}{f(\lambda)} &= \frac{a^{m_a -1}(m_a - 1)! \cdot (a + 1)^{m_{a + 1} + 1} (m_{a+1} + 1)! \cdot b^{m_b - 1} (m_b - 1)! \cdot (b - 1)^{m_{b - 1} + 1}(m_{b - 1} + 1)!}{a^{m_a }(m_a)! \cdot (a + 1)^{m_{a + 1}} (m_{a+1})! \cdot b^{m_b} (m_b)! \cdot (b - 1)^{m_{b - 1}}(m_{b - 1} )!} \\
&= \frac{(a +1) \cdot (m_{a + 1} + 1) \cdot (b - 1) \cdot (m_{b - 1} + 1)}{a \cdot m_a \cdot b \cdot m_b} = \frac{(a +1)  \cdot (b - 1) }{a \cdot m_a \cdot b \cdot m_b} < 1
\end{align*}

Since $a$ is the length of the top row and $b$ is the length of the bottom row, $m_{a + 1} = m_{b -1} = 0$, and the inequality follows since $m_a, m_b \ge 1$ and $b \le a$.

If $b = 1$, then $\ell(\lambda') = \ell(\lambda) - 1$ and
\[
\frac{f(\lambda')}{f(\lambda)} = \frac{  (a + 1)}{2 \cdot a m_a m_1} < 1.
\]
\end{proof}

\begin{corollary}
For $\lambda \vdash n$, $f(\lambda) \le 2n$ with equality if and only if $\lambda = (n)$.
\end{corollary}

\begin{remark}
It is natural to guess that $f(\lambda)$ is monotone in the usual partial order on partitions. This fails, for example:

\begin{center}
\begin{tikzpicture}
\draw (-.5, -.5) node {$\lambda = $ };
\foreach \x in {0,...,1}
	\filldraw (\x*.25, 0) circle (.5mm);
\foreach \x in {0,...,1}
	\filldraw (\x*.25, -.25) circle (.5mm);
\foreach \x in {0,...,1}
	\filldraw (\x*.25, -.5) circle (.5mm);
\foreach \x in {0,...,0}
	\filldraw (\x*.25, -.75) circle (.5mm);
\foreach \x in {0,...,0}
	\filldraw (\x*.25, -1.0) circle (.5mm);

\draw (2, -.5) node {$\lambda^\prime = $ };
\foreach \x in {0,...,1}
	\filldraw (\x*.25 + 2.75, 0) circle (.5mm);
\foreach \x in {0,...,1}
	\filldraw (\x*.25 + 2.75, -.25) circle (.5mm);
\foreach \x in {0,...,1}
	\filldraw (\x*.25 + 2.75, -.5) circle (.5mm);
\foreach \x in {0,...,1}
	\filldraw (\x*.25 + 2.75, -.75) circle (.5mm);

\end{tikzpicture}
\end{center}
Here $\lambda < \lambda'$, but $f(\lambda) = 2^9 \cdot 3! < 2^8\cdot 4! = f(\lambda')$.
Still, inspection of special cases suggest that the partial order in the lemma can be refined. 
\end{remark}

The lemma shows that the \emph{smallest} double coset in $B_n \backslash S_{2n} / B_n$ corresponds to $id, \lambda = 1^n$, and the largest corresponds to the $2n$ cycle $(1 2 \hdots 2n), \lambda = (n)$.

Dividing \eqref{eqn: Blam} by $|S_{2n}|$ gives the probability measure 
\begin{equation} \label{eqn: ewens1}
    P_n(\lambda) = Z^{-1} \cdot n! \cdot \frac{1}{2^{\ell(\lambda)} z_\lambda}, \,\,\,\,\,\,\, \text{with} \,\,\,\,\,\,\,\, Z = \frac{1}{2}\left( \frac{1}{2} + 1 \right) \hdots \left( \frac{1}{2} + n - 1 \right) 
\end{equation}
This shows that $P_n(\lambda)$ is the Ewens measure $P_\theta$ for $\theta = 1/2$. The Ewens measure with parameter $\theta$ is usually described as a measure on the symmetric group $S_n$ with
\begin{equation} \label{eqn: ewensSn}
    P_\theta(\eta) = \frac{\theta^{c(\eta)}}{\theta(\theta + 1) \dots (\theta + n - 1)}, \,\,\,\,\,\,\,\, c(\eta) = \text{number of cycles in  } \eta.
\end{equation}
If $\eta$ is in the conjugacy class corresponding to $\lambda \vdash n$, then $c(\eta) = \ell(\lambda)$ and the size of the conjucacy class is $n!/z_\lambda$. Using this, simple calculations show \eqref{eqn: ewens1} is \eqref{eqn: ewensSn} with $\theta = 1/2$.

The Ewens measure is perhaps the most well-studied non-uniform probability on $S_n$ because of its appearance in genetics. The survey by Harry Crane \cite{crane} gives a detailed overview of its many appearances and properties. Limit theorems for $P_\theta$ are well developed. Arratia-Barbour-Tavar\'{e} (\cite{ABT}, chapter 4) studies the distribution of cycles (number of cycles, longest and shortest cycles, etc.\ ) under $P_\theta$. The papers of F\'{e}ray \cite{feray} study exceedences, inversions, and subword patterns. A host of features display a curious property: The limiting distribution does not change with $q$ (!). For example, the structure of the descent set of an Ewens permutation matches that of a uniform permutation. An elegant, unified theory is developed in the papers of Kammoun \cite{kammounI}, \cite{kammounII}, \cite{kammounIII}. More or less any natural feature of $\lambda$ has been covered. These papers work for all $\theta$ so the results hold for $P_\theta$ in \eqref{eqn: ewensSn}.

The following section gives more details. The final section suggests related problems.

\subsection{Cycle indices and Poisson distributions} \label{sec: cycleIndex}

For a partition $\lambda$ of $n$ let $a_i(\lambda)$ be the number of parts of $\lambda$ equal to $i$. Thus, $\sum_{i = 1}^n i a_i(\lambda) = n$. For $\sigma \in S_{2n}$, write $a_i(\sigma)$ for $a_i(\lambda_\sigma)$ and introduce the generating functions:
\begin{align*}
    &f_n(x_1, \dots, x_n) = \frac{1}{(2n)!} \sum_{\sigma \in S_n} \prod_{i = 1}^n x_i^{a_i(\sigma)}, \,\,\,\,\,\,\, n \ge 1 \\
    &f_0 = 1
\end{align*}
and
\[
f(t) = \sum_{n = 0}^\infty t^n \frac{\binom{2n}{n}}{2^{2n}} f_n.
 \]
 The following analog of Polya's cycle index theorem holds.
 
 \begin{theorem} \label{thm: cycleIndexAnalog}
 With notation as above,
 \[
 f(t) = \exp \left( \sum_{n = 1}^\infty \frac{t^n}{2n} x_i \right) 
 \]
 \end{theorem}
 
 \begin{proof}
 The proof uses symmetric function theory as in \cite{macdonald}. In particular, the power sum symmetric functions in variables $y = (y_1, y_2, \dots)$ are $p_j(y) = \sum_i y_i^j$ and, for $\lambda = 1^{a_1} 2^{a_2} \hdots$, $p_\lambda(y) = \prod_i p_i^{a_i}$. A formula at the bottom of pg 307 in  \cite{macdonald} specializes to
 \begin{equation} \label{eqn: macdonald}
     \sum_\lambda z_\lambda^{-1} 2^{-\ell(\lambda)} p_\lambda(y) p_\lambda(y') = \exp \left( \sum_{n = 1}^\infty \frac{1}{2n} p_n(y) p_n(y') \right).
 \end{equation}
 In \eqref{eqn: macdonald}, $y$ and $y'$ are distinct sets of variables. We have set $v_\lambda = 2$ in Macdonald's formula (see the discussion following the proof). Set further $y = y'$ and replace $y$ by $\sqrt{t} y$ to get
 \[
  \sum_\lambda z_\lambda^{-1} 2^{-\ell(\lambda)} t^{|\lambda|} p_\lambda(y)^2 = \exp \left( \sum_{n = 1}^\infty \frac{t^n}{2n} p_n(y)^2 \right),
 \]
 where $|\lambda| = \sum_i \lambda_i$. Since the $p_n$ are free generators of the ring of symmetric functions, they may be specialized to $p_n \to \sqrt{x_i}$ (that is, setting $p_i(y) = \sqrt{x_i}$). Then the formula becomes
 \begin{equation} \label{eqn: 4.6}
     \sum_{n = 0}^\infty t^n \sum_{\lambda \vdash n} z_\lambda^{-1} 2^{-\ell(\lambda)} \prod_i x_i^{a_i(\lambda)} = \exp \left( \sum_{n = 1}^\infty \frac{t^n}{2n} x_n \right).
 \end{equation}
 As above, the inner sum is
 \[
 \sum_{\lambda \vdash n} z_\lambda^{-1} 2^{-\ell(\lambda)} \prod_i x_i^{a_i(\lambda)} = (2^n n!)^2 \sum_{\sigma \in S_{2n}} \prod_i x_i^{a_i(\lambda_\sigma)} = \frac{\binom{2n}{n}}{2^{2n}} f_n.
 \]

 \end{proof}

 To bring out the probabilistic content of Theorem \ref{thm: cycleIndexAnalog}, recall the negative binomial density with parameters $1/2, 1 - t$ assigns mass
 \[
 p_{1/2, 1 - t}(n) = Z^{-1} \frac{\binom{2n}{n}t^n}{2^{2n}}, \,\,\,\,\,\,\, Z^{-1} = \sqrt{1 - t}, \,\,\,\,\,\,\, n = 0, 1, 2, \dots 
 \]
 Divide both sizes of \eqref{eqn: 4.6} by $\sqrt{1 - t}$ to see
 \begin{equation} \label{eqn: negBinomIdentity}
     \sum_{n = 0}^\infty p_{1/2, 1-t}(n) \cdot f_n = \prod_{n = 1}^\infty \exp \left( \frac{t^n}{2n}x_n - \frac{t^n}{2n} \right), 
 \end{equation}
 using the expansion $\sqrt{1 - t} = \prod_n e^{t^n/2n}$. Recall the Poisson$(\lambda)$ distribution on $\{0, 1, 2, \dots, \}$ has density $e^{-\lambda} \lambda^j/j!$ and moment generating function $e^{-\lambda + \lambda x}$. This and \eqref{eqn: negBinomIdentity} gives
 
 \begin{corollary}
 Pick $n \in \{0, 1, 2, \dots \}$ from $p_{1/2, 1 -t}(n)$ and then $\sigma \in S_{2n}$ from the uniform distribution. If $\sigma$ has $\lambda_\sigma$ with $a_i$ parts equal to $i$, then the $\{a_i \}_{i  = 1}^n$ are independent with $a_i$ having a Poisson distribution with parameter $t^i/2i$.
 \end{corollary}
 
 From this corollary one may prove theorems about the joint distribution of cycles exactly as in \cite{sheppLloyd}. This gives analytic proofs of previously proved results. For example, for large $n$:
 \begin{itemize}
     \item The $\{a_i \}_{i = 1}^n$ are asymptotically independent with Poisson$(1/2i)$ distributions.
     
     \item $\ell(\lambda)$ has mean asymptotic to $\log(n)/2$, variance asymptotic to $\log(n)/2$, and normalized by its mean and variance $\ell(\lambda)$ has a limiting normal distribution.
 \end{itemize}
 
 The distribution of smallest and largest parts are similarly determined. The calculations in this section closely match the development in  \cite{watterson}. This gives a very clear description of the results above from the genetics perspective.

\subsection{Remarks and extensions}

\textbf{(a)} The formula of Macdonald used in Section \ref{sec: cycleIndex} involved a sequence of numbers $v_i, 1 \le i < \infty$. For a partition $\lambda$, define $v_\lambda = v_{\lambda_1} v_{\lambda_2} \hdots v_{\lambda_l}$ multiplicatively. Macdonald proves
\[
\sum_\lambda v_\lambda^{-1} z_\lambda^{-1} p_\lambda(y) p_{\lambda}(y') = e^{\sum_n p_n(y)p_n(y')/(nv_n)}.
\]
At the right, the product means `something is independent' and it is up to us to see what it is.

As a first example, take $v_i = 1$ for all $i$. Then, proceeding as in \eqref{eqn: negBinomIdentity} the formula becomes
\[
\sum_{n = 0}^\infty \frac{t^n}{n} C_n(x_1, \dots, x_n) = e^{\sum_{i = 1}^n x_i t^i/i},
\]
with $C_n(x_1, \dots, x_n) = \frac{1}{n!} \sum_{\sigma \in S_n} \prod_i x_i^{a_i(\sigma)}$, the cycle indicator of $S_n$. This is exactly Polya's cycle formula, see \cite{sheppLloyd}.

Taking $v_n \equiv 1/2$ gives the results of Section \ref{sec: cycleIndex} and indeed suggested the project of enumerating by double cosets. Macdonald considers the five following choices for $v_n$:
\[
1, (1 - t^n)^{-1}, 2, \alpha, (1 - q^n)/(1 - t^n)
\]
and shows that each gives celebrated special functions: Schur, Hall-Littlewood, Zonal, Jack, and Macdonald, respectively. We are sure that each will give rise to an interesting enumerative story, if only we could find out what is being counted. Indeed, in \cite{fulmanTori} Jason Fulman has shown that the case of $v_i = 1 - 1/q^i$ enumerates $F$-stable maximal tori in $GL_n(\overline{F})$.

\textbf{(b)} For the cycles of the symmetric group, Polya's formula shows that the limiting Poisson approximation is remarkably accurate. In particular, under the uniform distribution on $S_n$:
\begin{itemize}
    \item The first $n$ moments of the number of fixed points of $\sigma$, $a_1(\sigma)$, are equal to the first $n$ moments of the Poisson$(1)$ distribution.
    
    \item More generally, the mixed moments
    \[
    \E_{S_n}[a_1^{k_1} a_2^{k_2} \hdots a_{l}^{k_l}]
    \]
    equal the same moments of independent Poisson variables with parameters $1, 1/2, \dots, 1/l$, as long as $k_1 + 2k_2 + \hdots + lk_l \le n$.
\end{itemize}

Theorem \ref{thm: cycleIndexAnalog} allows exact computation of the joint mixed moments of $a_1, a_2, \dots$ for $\lambda$ chosen from Ewens$(1/2)$ distribution. They are not equal to the limiting moments. The moments were first computed by Watterson in \cite{watterson}.

\textbf{(c)} We mention a $q$-analog of the results of this section which is parallel and `nice'. It remains to be developed. The $n$-dimensional symplectic group $Sp_{2n}(\mathbb{F}_q)$ is a subgroup of $GL_{2n}(\mathbb{F}_q)$ and $GL_{2n}, Sp_{2n}$ is a Gelfand pair. The double cosets are nicely labeled and the enumerative facts are explicit enough that analogs of he results above should be applicable. For details, see \cite{bannaiKawanakaSong}.

Jimmy He (\cite{he2019characteristic}, \cite{he2020random}) worked out the convergence rates for the natural random walk on $GL_{2n}, Sp_{2n}$ using the spherical functions. This problem was suggested to the first author by Jan Saxl as a way of tricking himself into learning some probability. The result becomes a walk on quadratic forms, and He proves a cutoff occurs.

\section{Parabolic subgroups of $S_n$}

Let $\lambda$ be a partition of $n$ (denoted $\lambda \vdash n$). That is, $\lambda = (\lambda_1, \lambda_2, \dots, \lambda_I)$ with $\lambda_1 \ge \lambda_2 \ge \hdots \ge \lambda_I > 0$ and $\lambda_1 + \lambda_2 + \hdots + \lambda_I = n$. The parabolic subgroup $S_\lambda$ is the set of all permutations in $S_n$ which permute only $\{1, 2, \dots, \lambda_1 \}$ among themselves, only $\{\lambda_1 + 1, \dots, \lambda_1 + \lambda_2 \}$ among themselves, and so on. Thus,
\[
S_\lambda \cong S_{\lambda_1} \times S_{\lambda_2} \times \hdots \times S_{\lambda_I}.
\]
If $s_i = (i, i + 1), 1 \le i \le n -1$ are the generating transpositions of $S_n$, then $S_\lambda$ is generated by $\{s_i \}_{i = 1}^{n-1} \backslash \left\lbrace (s_{\lambda_1}, s_{\lambda_1 + 1}), \dots, (s_{n - \lambda_I - 1}, s_{n - \lambda_I} ) \right\rbrace$. The group $S_\lambda$ is often called a \emph{Young subgroup}.

Let $\mu = (\mu_1, \dots, \mu_J)$ be a second partition of $n$. This section studies the double cosets $S_\lambda \backslash S_n / S_\mu$. These cosets are a classical object of study; they can be indexed by contingency tables: $I \times J$ arrays of non-negative integers with row sums given be the parts of $\lambda$ and column sums the parts of $\mu$.

The mapping from $S_n$ to tables is easy to describe: Fix $\sigma \in S_n$. Inspect the first $\lambda_1$ positions in $\sigma$. Let $T_{11}$ be the number of elements from $\{1, 2, \dots, \mu_1 \}$ occurring in these positions, $T_{12}$ the number of elements from $\{\mu_1 + 1, \dots, \mu_1 + \mu_2 \}$, \dots and $T_{1J}$ the number of elements from $\{n - \mu_{J} + 1, \dots, n \}$. In general, $T_{ij}$ is the number of elements from $\{\mu_1 + \hdots + \mu_{i -1} + 1, \dots, \mu_1 + \hdots + \mu_j \}$ which occur in the positions $\lambda_1 + \lambda_2 + \hdots + \lambda_{i - 1} + 1$ up to $\lambda_1 + \hdots + \lambda_i$.

\begin{example} \label{ex: sigmaT}
When $n = 5$, $\lambda = (3, 2)$, $\mu = (2, 2, 1)$ there are five possible tables:
\begin{align*}
& \begin{pmatrix}
2 & 1 & 0 \cr 
0 & 1 & 1
\end{pmatrix} \,\,\,\,\,\,\,\,\,\,\,  \begin{pmatrix}
2 & 0 & 1 \cr 
0 & 2 & 0
\end{pmatrix}
\,\,\,\,\,\,\,\,\,\,\, \begin{pmatrix}
1 & 2 & 0 \cr 
1 & 0 & 1
\end{pmatrix}
\,\,\,\,\,\,\,\,\,\,\,  \begin{pmatrix}
1 & 1 & 1\cr 
1 & 1 & 0
\end{pmatrix}
\,\,\,\,\,\,\,\,\,\,\,  \begin{pmatrix}
0 & 2 & 1 \cr 
2 & 0 & 0
\end{pmatrix} \\
&\hspace{4mm} \sigma = 12345  \hspace{10mm} \sigma = 12543  \hspace{10mm} \sigma = 13425 \hspace{10mm} \sigma = 13524 \hspace{10mm} \sigma = 34512 \\
&\,\,\,\,\,\,\,\,\,\,\,\, 24  \hspace{23mm} 12  \hspace{23mm} 24 \hspace{22mm} 48 \hspace{22mm} 12
\end{align*}
Listed below each table is a permutation in the corresponding double coset, and the total size of the double cosest.
\end{example}

The mapping $\sigma \to T(\sigma)$ is $S_\lambda \times S_\mu$ bi-invariant and gives a coding of the double cosets. See \cite{jamesKerber} for further details and proof of this correspondence. Jones \cite{jones} gives a different coding. 

Any double coset has a unique minimal length representative. This is easy to identify: Given $T$, build $\sigma$ sequentially, left to right, by putting down $1, 2, \dots, T_{11}$ then $\mu_1 + 1, \mu_1 + 2, \dots, \mu_1 + T_{12}$ ... each time putting down the longest available numbers in the $\mu_j$ block, in order. Thus, in example \ref{ex: sigmaT} the shortest double coset representative is $13524$. For more details, see \cite{billeyKonvPeterSlofstra}.

The measure induced on contingency tables by the uniform distribution on $S_n$ is
\begin{equation} \label{eqn: fisherNew}
    P_{\lambda, \mu}(T) = \frac{1}{n!} \prod_{i, j} \frac{\lambda_i! \mu_j!}{T_{ij}!}.
\end{equation}
This is the \emph{Fisher-Yates} distribution on contingency tables, a mainstay of applied statistical work in chi-squared tests of independence. The distribution can be described by a \emph{sampling without replacement} problem: Suppose that an urn contains $n$ total balls of $I$ different colors, $r_i$ of color $i$. To empty the urn, make $J$ sets of draws of unequal sizes.  First draw $c_1$ balls, next $c_2$, and so on until there are $c_J = n - \sum_{j = 1}^{J - 1} c_j$ balls left. Create a contingency table by setting $T_{ij}$ to be the number of color $i$ in the $j$th draw.

This perspective, along with the previously defined mapping from permutations to cosets, proves that the distribution on contingency tables induced by the uniform distribution on $S_n$ is indeed the Fisher-Yates: Suppose a permutation $\sigma \in S_n$ represents a deck of cards labeled $1, \dots, n$. Given partitions $\lambda, \mu$ color cards $1, \dots, \mu_1$ with color 1, labels $\mu_1 + 1, \dots, \mu_2$ color $2$ and so on. From a randomly shuffled deck, draw the first $\lambda_1$ cards and count the number of each color, then draw the next $\lambda_2$, and so on.

More statistical background and available distribution theory is given in the following section. These results give some answers to the question:
\begin{equation} \label{eqn: question}
    \text{Pick } \, \sigma \in S_n \,\, \text{uniformly. What} \,\,\, S_\lambda \backslash S_n / S_\mu \,\,\, \text{double coset is it likely to be in?}
\end{equation}

From \eqref{eqn: fisherNew},
\begin{equation}
    |S_\lambda \sigma S_\mu| = \prod_{i, j} \frac{\mu_i! \lambda_j!}{T_{ij}!}, \,\,\,\, \text{for} \,\,\, T = T(\sigma).
\end{equation}
However, enumerating the \emph{number} of double cosets is a \#-P complete problem. See \cite{DG}.

When $\lambda = \mu = (k, n - k)$, the double cosets give a Gelfand pair with spherical functions the Hahn polynomials. The associated random walk is the \emph{Bernoulli-Laplace urn}, which is perhaps the first Markov chain! (See \cite{Dshashahani}.) More general partitions give interesting urn models but do not seem to admit orthogonal polynomial eigenvectors.

One final note: there has been a lot of study on the \emph{uniform distribution} on the space of tables with fixed row and column sums. This was introduced with statistical motivation in \cite{DEfron}. The central problem has been efficient generation of such tables; enumerative theory is also natural but remains to be developed. See \cite{DG}, \cite{CDHL}, \cite{dittmer2019thesis}, \cite{DittmerLyuPak}, \cite{barvinok2008} and their references. The Fisher-Yates distribution \eqref{eqn: fisherNew} is quite different from the uniform and central to both the statistical applications and to the main pursuits of the present paper.

Section \ref{sec: statContingency} develops statistical background and uses this to understand the size of various double cosets, Section \ref{sec: zeros} proves a new limit theorem for the number of zeros in $T(\sigma)$. The final section discusses natural open problems. 

\subsection{Statistical background} \label{sec: statContingency}

Contingency tables arise whenever a population of size $n$ is classified with two discrete categories. For example, Table 1 shows 592 subjects classified by 4 levels of eye color and 4 levels of hair color.

\begin{table}[] \label{tab: data}
\begin{center}

\begin{tabular}{l|llll|l}
               & Black & Brown & Red & Blond & \textbf{Total} \\ \hline
Brown          & 68    & 119   & 26  & 7     & 220            \\
Blue           & 20    & 84    & 17  & 94    & 215            \\
Hazel          & 15    & 54    & 14  & 10    & 93             \\
Green          & 5     & 29    & 14  & 16    & 64             \\ \hline
\textbf{Total} & 108   & 286   & 71  & 127   & \textbf{592}  
\end{tabular}

\end{center}

\caption{This table has a total of 592 entries, with row sums $r_1, r_2, r_3, r_4 = 220, 215, 93, 64$ and column sums $c_1, c_2, c_3, c_4 = 108, 286, 71, 127$. There are $1, 225, 914, 276, 768, 514 \doteq 1.225 \times 10^{15}$ tables with these row and column sums. }
\end{table}


A classic task is the chi-squared test for independence. This is based in the chi-squared statistic
\begin{equation} \label{eqn: chiSquareStat}
    \chi^2(T) = \sum_{i, j} \left(T_{ij} - \frac{\lambda_i \cdot \mu_j}{n} \right)^2 / (\lambda_i \cdot \mu_j/n).
\end{equation}
This measure how close the table is to a natural product measure on tables. In the example Table 1, $\chi^2 = 138.28$.

The usual probability model for such tables considers a population of size $n$, with each individual independently assigned into one of the $I \times J$ cells with probability $p_{ij}$ ( $p_{ij} \ge 0, \sum_{ij} p_{ij} = 1$). The \emph{independence model} postulates
\[
p_{ij} = \alpha_i \cdot \beta_j
\]
for $\alpha_i, \beta_j \ge 0$ and $\sum_i \alpha_i = \sum_j \beta_j = 1$. A basic theorem in the subject \cite{kangKlotz} says that if $n$ is large and $\alpha_i, \beta_j > 0$ the $\chi^2$ statistic has a limiting distribution $f_k(x)$, i.e.\ 
\[
P(\chi^2 \le x) \to \int_0^x f_k(t) \, dt,
\]
where $f_k(x)$ is the chi-squared density with $k = (I - 1) (J - 1)$ degrees of freedom:
\begin{equation} \label{eqn: chiSquareDensity}
    f_k(x) = \frac{x^{k/2 - 1} \cdot e^{-x/2}}{2^{k/2} \cdot \Gamma(k/2)}, \,\,\,\,\, x \ge 0.
\end{equation}
The density $f_k$ has mean $k$ and variance $2k$ and it is customary to compare the observed $\chi^2$ statistic with the $k \pm \sqrt{2k}$ limits and reject the null hypothesis if the statistic falls outside this interval. In the example, $k = 9$ and the hypothesis of independence is rejected.

The above simple rendition omits many points which are carefully developed in \cite{lancaster}, \cite{agresti92}, \cite{agrestiBook}.

The great statistician R.A.\ Fisher suggested a different calibration: Fix the row sums, fix the column sums and look at the conditional distribution of the table given the row and column sums (under the independence model). It is an elementary calculation to show that $\P(T \mid \lambda_i, \mu_j)$ is the Fisher-Yates distribution \eqref{eqn: fisherNew}. Notice that the Fisher-Yates distribution does not depend on the `nuisance parameters' $\alpha_i, \beta_j$. This is called \emph{Fishers exact test}. There is a different line of development leading to the same distribution. This is the conditional testing approach (also due to Fisher). David Freedman and David Lane \cite{freedmanLaneI}, \cite{freedmanLaneII} give details, philosophy, and history. We only add that conditional testing is a rich, difficult subject (starting with the question: what to condition on?). For discussion and extensive pointers to the literature, see \cite{lehmannRomanoBook} (Chapter 2), \cite{Dsturfmels} (Section 4). 

All of this said, mathematical statisticians have long considered the distribution of tables with given row and column sums under the Fisher-Yates distribution.

The following central limit theorem determines the joint limiting distribution of the table entries $T_{ij}$ under the Fisher-Yates distribution. They are approximately multivariate normal. As a corollary, the $\chi^2$ statistic has the appropriate chi-squared distribution. This can be translated into estimates of the size of various double cosets, as discusses after the statement.

In the following, fix $I$ and $J$. Let $\lambda^n = ( \lambda_1^n, \dots, \lambda_I^n), \mu^n = (\mu_1^n, \dots, \mu_J^n)$ be two sequences of partitions of $n$. Suppose there are constants $\alpha_i, \beta_j$ with $0 < \alpha_i, \beta_j < 1$ such that
\begin{equation}
\label{eqn: alphaBetaAssumption}
    \lim_{n \to \infty} \lambda_i^n/n = \alpha_i, \,\,\,\,\, \lim_{n \to \infty} \mu_j^n/n = \beta_j \,\,\,\,\, \text{for} \,\,\, 1 \le i \le I, 1 \le j \le J.
\end{equation}
Let $T$ be drawn from the Fisher-Yates distribution \eqref{eqn: fisherNew} and let
\[
Z_{ij}^n = \sqrt{n} \left( \frac{T_{ij}}{n} - \frac{\lambda_i^n \mu_j^n}{n^2} \right)
\]

\begin{theorem} \label{thm: tableCLT}
With notation as above, assuming \eqref{eqn: alphaBetaAssumption}, the random vector
\[
Z_n = (Z_{11}^n, Z_{12}^n, \dots, Z_{1J}^n, \dots, Z_{I1}^n, \dots, Z_{IJ}^n)
\]
converges in distribution to a normal distribution with mean zero and covariance matrix
\[
\Sigma = \left( \text{Diag}(\alpha) - \alpha \cdot \alpha^T \right) \otimes \left( \text{Diag}(\beta) - \beta \cdot \beta^T \right),
\]
for $\alpha = (\alpha_1, \dots, \alpha_I)$, $\beta = (\beta_1, \dots, \beta_J)$.
\end{theorem}

The tensor product in the definition of $\Sigma$ means that covariance between the $i_1, j_1$ variable and the $i_2, j_2$ variable is given by $\left( \text{Diag}(\alpha) - \alpha \cdot \alpha^T \right)_{i_1, i_2} \cdot \left( \text{Diag}(\beta) - \beta \cdot \beta^T \right)_{j_1, j_2}$. Note that since the final entry in each row (or column) is determined by the other entries, the $IJ \times IJ$ covariance matrix is singular with rank $(I-1)(J-1)$.

\begin{corollary} \label{cor: chi}
Under the conditions of Theorem \ref{thm: tableCLT}, the chi-squared statistic \eqref{eqn: chiSquareStat} has a limiting chi-squared distribution \eqref{eqn: chiSquareDensity} with $k = (I - 1)(J - 1)$ degrees of freedom.
\end{corollary}

A very clear proof of Theorem \ref{thm: tableCLT} and the corollary is given by Kang and Klotz \cite{kangKlotz}. They review the history, as well as survey several approaches to the proof. Their argument is a classical, skillful use of Stirling's formula and their paper is a model of exposition.

The usual way of using these results, for a single entry $T_{ij}$ in the table, gives
\[
\P \left( \frac{T_{ij} - \nu_{ij}}{\sqrt{n \sigma_{ij}^2}} \le x \right) \sim \frac{1}{\sqrt{2 \pi}} \int_{-\infty}^x e^{-t^2/2} \, dt, \,\,\,\,\, \nu_{ij} = \frac{\lambda_i \mu_j}{n}, \,\,\, \sigma_{ij}^2 = \frac{\lambda_i \mu_j}{n^2}\left(1 - \frac{\lambda_i \mu_j}{n^2} \right).
\]

Any single entry of the table has a limiting normal approximation. This can also be seen through the normal approximation to the hypergeometric distribution. This is available with a Berry-Esseen error; see \cite{hoglund}.

The limiting $\chi^2$ approximation shows that, under the Fisher-Yates distribution, most tables are concentrated around the `independence table'
\[
T_{ij}^* = \frac{\lambda_i \mu_j}{n}.
\]
This $T^*$ is rank one. While it does not have integer entries, it gives a good picture of the approximate size of a typical double coset. 

To be quantitative, let us define a distance between tables $T, T'$ with the same row and column sums:
\[
\|T - T' \|_1 = \sum_{i, j} |T_{ij} - T_{ij}'|.
\]
This is the $L^1$ distance, familiar as total variation from probability. Since $\sum_{i, j} T_{ij} = n$, for many tables $T, T'$, $\|T - T' \|_1 \doteq n$. The Cauchy-Schwartz inequality shows
\begin{equation} \label{eqn: CS}
    \|T - T^* \|_{1} \le \sqrt{n} \cdot \chi^2(T).
\end{equation}
Corollary \ref{cor: chi} shows that, under the Fisher-Yates distribution, $\chi^2(T)$ is typically $(I - 1)(J-1) \pm \sqrt{2(I - 1)(J - 1)}$, and thus typically $\|T - T^*\|$ is of order $\sqrt{n} \ll n$. A different way to say this is to divide the tables $T$ and $T^*$ by $n$ to get probability distributions $\overline{T}, \overline{T}^*$ on $IJ$ points. Then, for most $T$,
\[
\| \overline{T} - \overline{T}^* \|_1 = O_p \left( \frac{1}{\sqrt{n}} \right).
\]

Barvinok \cite{barvinok2008} studies the question in the paragraph above under the \emph{uniform} distribution on tables. In this setting, he shows that most tables are close (in a somewhat strange distance) to quite a different table $T^{**}$.

Theorem \ref{thm: tableCLT} also gives an asymptotic approximation to the size of the double coset corresponding to the table $T$. Call this $S_\lambda T S_\mu$. It is easy to see that $|S_\lambda T S_\mu| = n!\cdot  \P(T \mid \lambda, \mu) \sim n! \cdot \varphi(T)/\sqrt{n}$ with
\[
\varphi(T) = \frac{e^{-Z_{-}\Sigma_{-}^{-1} Z_{-}^T/2}}{\text{det}(\Sigma_{-})^{1/2}}
\]
for $Z_{-}$ the vector corresponding to the upper left $(I - J) \times (J - 1)$ sub-matrix of $T$ (with notation as in Theorem \ref{thm: tableCLT}) and $\Sigma_{-}$ the associated $(I - 1)(J - 1) \times (I - 1)(J-1)$ covariance matrix (that is, the covariances between the remaining $(I - 1)\cdot(J-1)$ entries of the sub-matrix). Note that removing one row and one column from $T$ removes the dependency so $\Sigma_{-}$ is full rank. This uses the local limit version of Theorem \ref{thm: tableCLT}, which follows from the argument of Kang and Klotz \cite{kangKlotz}. See \cite{chagantySethuraman} for further details.

The asymptotics above show that the large double cosets are the ones closest to the independence table. This may be supplemented by the following non-asymptotic development.

Let $T$ and $T'$ be tables with the same row and column sums. Say that $T \prec T'$ (`$T'$ majorizes $T$') if the largest element in $T'$ is greater than the largest element in $T$, the sum of the two largest elements in $T'$ is greater than the sum of the two largest elements in $T$, and so on. Of course the sum of all elements in $T'$ equals the sum of all elements of $T$.

\begin{example} \label{ex: major}
For tables with $n = 8, \lambda_1 = \lambda_2 = \mu_1 = \mu_2 = 4$, there is the following ordering
\[
\begin{pmatrix}
2 & 2 \cr 
2 & 2 
\end{pmatrix} \prec \begin{pmatrix}
3 & 1 \cr 
1 & 3 
\end{pmatrix} \prec
\begin{pmatrix}
4 & 0 \cr 
0 & 4 
\end{pmatrix}.
\]
\end{example}

Majorization is a standard partial order on vectors \cite{marshallOlkinBooks} and Harry Joe \cite{joe1985ordering} has shown it is useful for contingency tables.

\begin{proposition}
Let $T$ and $T'$ be tables with the same row and column sums and $P$ the Fisher-Yates distribution. If $T \prec T'$, then \[
P(T) > P(T').
\]
\end{proposition}

\begin{proof}
From the definition \eqref{eqn: fisherNew}, we have $\log(P(T)) = C - \sum_{i, j} \log(T_{ij}!)$ for a constant $C$. This form makes it clear the right hand side is a symmetric function of the $IJ$ numbers $\{T_{ij} \}$. The log convexity of the Gamma function shows that it is concave. A symmetric concave function is Schur concave: That is, order-reversing for the majorization order \cite{marshallOlkinBooks}.
\end{proof}

\begin{remark}
Joe \cite{joe1985ordering} shows that, among the real-valued tables with given row and column sums, the independence table $T^*$ is the unique smallest table in majorization order. He further shows that if an integer valued table is, entry-wise, within $1$ of the real independence table, then $T$ is the unique smallest table with integer entries. In this case, the corresponding double coset has $P(T)$ largest. 
\end{remark}

\begin{example}
Fix a positive integer $a$ and consider an $I \times J$ table $T$ with all entries equal to $a$. This has constant row sums $J\cdot a$ and column sums $I \cdot a$. It is the unique smallest table with these row and column sums, and so corresponds to the largest double coset. For $a = 2, I = 2, J = 3$, this table is
\[
T = \begin{pmatrix} 2 & 2 & 2 \cr 2 & 2 & 2 \end{pmatrix}.
\]
\end{example}

Contingency tables with fixed row and column sums form a graph with edges between tables that can be obtained by one move of the following: pick two rows $i, i'$ and two columns $j, j'$. Add $+1$ to the $(i, j)$ entry, $-1$ to the $(i', j)$ entry,  $+1$ to the $(i', j')$ entry, and $-1$ to the $(i, j')$ entry. This graph is connected and moves up or down in the majorization order as the $2 \times 2$ table with rows $i, i'$ and columns $j, j'$ moves up or down. See Example \ref{ex: major} above.

\subsection{Zeros in Fisher-Yates tables} \label{sec: zeros}

In this section we will use $r_1, \dots, r_I$ for the row sums of a table and $c_1, \dots, c_J$ for the column sums. One natural feature of a contingency table is its zero entries. As shown in Section \ref{sec: statContingency}, most tables will be close to the table $T^*$ with entries $T^*_{ij} = r_i c_j/n$. This has \emph{no} zero entries. Therefore, zeros are a pointer to the breakdown of the independence model. In statistical applications, there is also the issue of `structural zeros' -- categories such as `pregnant males' which would give zero entries in cross-classified data due their impossibility. See \cite{bishop2007discrete} for discussion. The bottom line is, professional statisticians are always on the look-out for zeros in contingency tables. This section gives a limit theorem for the number of zeros under natural hypotheses.

A simple observation which leads to the theorem is that a Fisher-Yates table is equivalent to rows of independent multinomial vectors, conditioned on the column sums: let $X_1, \dots, X_I$ be independent random vectors of length $J$, with $X_i \sim Multinomial(r_i, \{q_j \}_{j = 1}^J)$ for some probabilities $q_j > 0$ and $\sum_j q_j = 1$. That is, $X_i$ are the occupancy counts generated by assigning $r_i$ balls to $J$ boxes, with one ball going to the $j$th box with probability $q_j$. The joint distribution for the vectors is then
\begin{equation} \label{eqn: multiJoint}
\P(X = (x_{ij}) ) = \prod_{i = 1}^I  \binom{r_i}{x_{i1}, \dots, x_{iJ}} \cdot  q_1^{x_{i1}} \hdots q_J^{x_{iJ}}.
\end{equation}

Let $Y_1, \dots, Y_I$ be distributed as $X_1, \dots, X_I$ conditioned on the sums $\sum_{i = 1}^I X_{ij} = c_j$. From \eqref{eqn: multiJoint} it is clear that $Y_1, \dots, Y_I$ has the Fisher-Yates distribution \eqref{eqn: fisherNew}, regardless of the choices $q_j$.

This perspective allows us to use known limit results for multinomial distributions, translated to contingency tables using conditioned limit theory. For the remainder, assume that the row sums $r_i = r$ are constant, so that the $X_i$ are iid vectors. Let $f(X_i) = \sum_{j = 1}^J \textbf{1}(X_{ij} = 0)$ count the number of zero-entries in the vector. \cite{holstUrn} contains limit theorems for $f(X_i)$ as $r \to \infty$, with either Poisson or normal limit behavior depending on the asymptotics of $r, J$ and the $q_j$.

\begin{example}

Consider an $I \times J$ table with constant column sums $c = I(\log(I \cdot J) + \theta)$. The row sums are determined by $r = n/I$. If the table is created from the counts of dropping $n$ balls in $I \cdot J$ boxes, with each box equally likely, then the expected number of zero entries is
\[
\lambda^* = IJ \cdot \left(1 - \frac{1}{I \cdot J} \right)^n = IJ \cdot \left(1 - \frac{c}{n \cdot J} \right)^n \sim IJ e^{-c/J} \sim e^{-\theta}
\]
If $n, I, J \to \infty$ then $\lambda^* \to e^{- \theta}$, and the following theorem shows that the number of zeros has a Poisson$(\lambda^*)$ distribution under these assumptions. Indeed, it shows this for varying column sums.

\end{example}

\begin{theorem} \label{thm: zeros}
Suppose that $n \to \infty$ and fix sequences $I_n, J_n, c_j^n$ such that
\[
I_n \cdot \sum_{j = 1}^{J_n} \left( 1 - \frac{c_j^n}{n} \right)^{n/I_n}  \to \beta.
\]
Let $Z_n$ be the number of zeros in a Fisher-Yates contingency table of size $I_n \times J_n$ with constant row sums $r^n = n/I^n$ and column sums $c_1^n, \dots, c_{J_n}^n$. Then 
\[
\mathcal{L}\left( Z_n \right) \to Poisson(\beta).
\]
\end{theorem}

\begin{proof}
Let $X_j^n \sim$ Multinomial$(r^n, \{q_j^n \}_{j = 1}^{J_n} )$, with the probabilities $q_j = c_j^n/n$ chosen so that $\E[\sum_{i = 1}^{I_n} X_{ij}^n] = I_n \cdot r^n \cdot q_j^n = c_j^n$. Then conditioned limit theorem (Corollary 3.5 in \cite{holstExponential}) says that if
\[
\mathcal{L}\left( \sum_{i = 1}^{I_n} f(X_i^n) \right) \to \mathcal{L} (U),
\]
where $U$ has no normal component, then
\[
\mathcal{L}\left( \sum_{i = 1}^{I_n} f(Y_i^n) \right) = \mathcal{L}\left(  \sum_{i = 1}^I f(X_i^n) \mid \sum_{i = 1}^{I_n} X_{ij}^n = c_j^n, 1 \le j \le J_n \right) \to \mathcal{L} (U).
\]
If $X$ is a multinomial generated by dropping $r$ balls in $J$ boxes, with probabilities $q_j$, and if 
\[
\sum_{j = 1}^J \left(1 - q_j \right)^r \to \alpha,
\]
then the number of empty boxes is asymptotically Poisson$(\alpha)$ (e.g.\ Theorem 6D in \cite{barbourHolstJanson}). Thus the condition
\[
I^n \sum_{j = 1}^{J_n} \left( 1 - \frac{c_j^n}{r^nI_n} \right)^{r^n} \to \beta
\]
means that $f(X_i^n)$ is asymptotically Poisson$(\beta/I^n)$ and so $\sum_{i = 1}^{I^n} f(X_i^n)$ is Poisson$(\beta)$.

\end{proof}

Preliminary computations indicate that Theorem \ref{thm: zeros} will hold with row sums that do not vary too much. Figure \ref{fig: poisson} shows the result from simulations for the number of zeros in a $50 \times 20$ table with row and column sums fixed.

\begin{figure}[ht] \label{fig: poisson}
\begin{center}

\includegraphics[scale = 0.4]{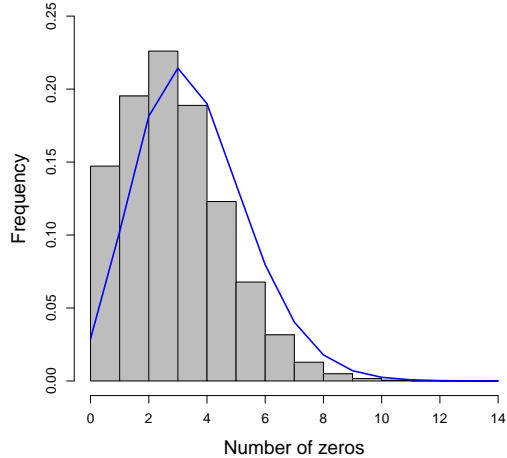}

\end{center}
\caption{Results for the number of zeros from 50,000 samples of a contingency table with $I = 50,
J = 20, c = 275, r = 110$. The blue curve is the frequency polygon of a Poisson distribution with $\lambda = 3.54$.}

\end{figure}


\subsection{Further questions} \label{sec: 5conclusion}

It is natural to ask further questions about the distribution of natural features of the tables representing double cosets. Three that stand out:
\begin{enumerate}
    \item The \emph{positions} of the zeros under the hypotheses of Theorem \ref{thm: zeros}.
    
    \item The size and distribution of the maximum entry in the table.
    
    \item The RSK shape: Knuth's extension of the Robinson-Schensted correspondence assigns to a table $T$ a pair $P, L$ of semi-standard Young tableux of the same shape. We have not seen these statistics used in statistical work. So much is known about RSK asymptotics that this may fall out easily. 
    
    \item Going back to Section 1: One nice development in probabilistic group theory on the symmetric group has been to look at the distribution of natural statistics within a fixed conjugacy class \cite{fulmanKimLee}, \cite{kimLee}. In parallel, one could fix a double coset and look at the distribution of standard statistics.

    \item Going further, this section has focused on enumerative probabilistic theorems for parabolic subgroups of the symmetric group. The questions make sense for parabolic subgroups of any finite Coxeter group. An enormous amount of combinatorial description is available (how does one describe double cosets?). This is wonderfully summarized in the very accessible paper \cite{billeyKonvPeterSlofstra}. In any Coxeter group, each double coset contains a unique minimal length representative. These minimal length double coset representatives can be used as identifiers for the double coset. See \cite{HeMinimal} for more on this. The focus of \cite{billeyKonvPeterSlofstra} is understanding $W_S \cdot \omega \cdot W_T$ with $\omega$ fixed as $S$ and $T$ vary over subsets of the generating reflections.
\end{enumerate}

\paragraph{Acknowledgements}
We thank Jason Fulman, Bob Guralnick, Jimmy He, Marty Isaacs, Slim Kammoun, Sumit Mukherjee, Arun Ram, Mehrdad Shahshahani, Richard Stanley, Nat Theim, and Chenyang Zhong for their help with this paper. MS is supported by a National Defense Science \& Engineering graduate fellowship. Research supported in part by National Science Foundation grant DMS 1954042.

\bibliography{citations}

\end{document}